\documentclass[reqno]{amsart}
\usepackage{amsfonts}
\usepackage{amsmath,amsthm,mathrsfs}
\usepackage{amssymb}
\usepackage{enumitem}
\usepackage{tikz}
\usepackage{bm}

\usepackage[colorlinks=true,linkcolor=blue,citecolor=blue,urlcolor=blue,pdfborder={0 0 0}]{hyperref}

\usepackage{soul}

\usepackage{graphicx}
\usepackage{subcaption}

\usepackage{xcolor}
\usepackage[capitalize]{cleveref}

\theoremstyle{plain}
\newtheorem{thm}{Theorem}[section]
\newtheorem{prop}[thm]{Proposition}
\newtheorem{defin}[thm]{Definition}
\newtheorem{lemma}[thm]{Lemma}

\newtheorem{cor}[thm]{Corollary}
\newtheorem{claim}[thm]{Claim}

\theoremstyle{definition}

\usepackage{fancyvrb}

\newcommand{\e}{\varepsilon}

\newcommand{\Reff}{R_\mathrm{eff}}

\DeclareMathOperator{\Exp}{\mathbb{E}}
\DeclareMathOperator{\E}{\mathbb{E}}

\DeclareMathOperator{\Prob}{\mathbb{P}}

\newcommand{\stab}{\mathrm{Stab}}
\newcommand{\T}{\Gamma}
\newcommand{\Gb}{\mathcal{G}_\bullet}

\newcommand{\lr}{\leftrightarrow}

\newcommand{\eq}{equation}
\newcommand{\be}{\begin{\eq}}
\newcommand{\ee}{\end{\eq}}

\renewcommand{\and}{\hbox{ and }}

\newtheorem{manualtheoreminner}{Theorem}
\newenvironment{manualtheorem}[1]{%
  \renewcommand\themanualtheoreminner{#1}%
  \manualtheoreminner
}{\endmanualtheoreminner}

\newtheorem{manuallemmainner}{Lemma}
\newenvironment{manuallemma}[1]{%
  \renewcommand\themanuallemmainner{#1}%
  \manuallemmainner
}{\endmanualtheoreminner}

\newtheorem{manualcorinner}{Corollary}
\newenvironment{manualcor}[1]{%
  \renewcommand\themanualcorinner{#1}%
  \manualcorinner
}{\endmanualtheoreminner}

\title[The local limit of uniform spanning trees]{The local limit of uniform spanning trees}
\author{Asaf Nachmias and Yuval Peres}

\AtEndDocument{\bigskip{\footnotesize%
\par

\textsc{Asaf Nachmias} \par
\textsc{Department of Mathematical Sciences, Tel Aviv University, Tel Aviv 69978, Israel} \par
  \textit{Email:} \texttt{asafnach@tauex.tau.ac.il} \par
  \addvspace{\medskipamount}

\textsc{Yuval Peres} \par
\textit{Email:} \texttt{yperes@gmail.com} \par
\addvspace{\medskipamount}

}}

\begin{document}

\begin{abstract} We show that the local limit 
of the uniform spanning tree on any finite, simple, connected, regular graph sequence with degree tending to $\infty$ is the Poisson$(1)$ branching process conditioned to survive forever. An extension to ``almost'' regular graphs and a quenched version are also given.
\end{abstract}

\maketitle

\vspace{.2cm}

\section{Introduction}

A {\bf spanning tree} $T$ of a finite connected graph $G$ is a subset of edges spanning a connected graph, containing no cycles, and such that every vertex of $G$ is incident to some edge of $T$. The {\bf uniform spanning tree} (UST) of such a graph $G$ is a uniformly drawn tree from the finite set of spanning trees of $G$. In this paper we study the \emph{local limit} of the UST on regular connected graphs with large degree. This limit, defined by Benjamini and Schramm \cite{BeSc}, is an infinite random rooted tree which encodes the local structure of the UST viewed from a typical vertex, see further definitions and discussion below.

\begin{thm}\label{thm:mainthm} Let $\{G_n\}$ be a sequence of finite, simple, connected, regular graphs with degree $d(n)\to \infty$. Then the local limit of the UST on $G_n$ is the Poisson$(1)$ branching process conditioned to survive forever.
\end{thm}
The large scale geometry of the UST on regular graphs of high degree may behave very differently depending on the underlying graph, see \cref{sec:maxdiamquestion}. It is therefore surprising that from the local point of view the UST's behavior is universal.




Let us present an equivalent yet more concrete version of this theorem giving precisely the probability of seeing a fixed rooted tree in a ball of radius $r$ around a random vertex. Let $T$ be a finite rooted tree of height $r \geq 1$, that is, the maximal graph distance between the root and  a vertex of $T$ equals $r$. Denote by $T_r$ the set of vertices of $T$ at graph distance precisely $r$ from the root and by $\stab_T$ the set of graph automorphisms of $T$ that preserve the root. Given a tree $\T$, an integer $r \geq 0$ and a vertex $v$ we write $B_{\T}(v,r)$ for the induced rooted subtree of $\T$ on the vertices that are at graph distance at most $r$ from $v$ in $\T$. \cref{thm:mainthm} can be now restated as follows.

\begin{thm}\label{thm:mainthm2} Let $\{G_n\}$ be a sequence of finite, simple, connected, regular graphs with degree $d(n)\to \infty$. Let $\T_n$ be a UST of $G_n$ and $X$ a uniformly chosen vertex of $G_n$. Then for any fixed rooted tree $T$ of height $r \geq 1$ we have that
$$ \lim_{n\to\infty} \Prob ( B_{\T_n}(X, r) \cong T) = \frac{|T_r| e^{-|V(T)|+|T_r|}}{|\stab_T|} \, ,$$
where $\cong$ means rooted graph isomorphism. 
\end{thm}

We also provide a stronger version of \cref{thm:mainthm2} (\cref{thm:mainthm2Robust}) in which the assumption of regularity can be weakened to ``almost regular'' graphs of high degree which include small perturbations of regular graphs and graphs in which the degree of most vertices are almost equal. Furthermore, in this more general setting we obtain a quenched version of \cref{thm:mainthm2} (\cref{thm:mainthm2quenched}) showing that with probability $1-o(1)$ the number of appearances of a fixed rooted tree $T$ in the UST $\T_n$ is 
$$ (1+o(1)) \frac{|V(G_n)||T_r| e^{-|V(T)|+|T_r|}}{|\stab_T|} \, .$$
For instance, with high probability the density of leaves in $\T_n$ is $(1+o(1))e^{-1}$ and the density of leaves attached to a vertex of degree $k$ is $(1+o(1))e^{-2}/(k-1)!$. These extensions are presented in \cref{sec:extensions}. \\

When $G$ is the complete graph on $n$ vertices the conclusion of \cref{thm:mainthm} was established in the pioneering work of Grimmett \cite{Gri:RandomTree}. The local limit of USTs on dense graphs (i.e., when the number of edges is proportional to the number of vertices squared) was recently investigated in \cite{HNT18} where it was shown to be a certain multi-type branching process that can be described explicitly via the limiting graphon.  The assumption of density of the graph is used in \cite{HNT18} to apply partition theorems in the spirit of Szemer\'edi's Lemma and to use ``graph-limit'' techniques to explicitly describe the local limit. These techniques are unavailable in the setting of this paper where the degree tends to infinity arbitrarily slow.

We remark that \cref{thm:mainthm} recovers the result of \cite{HNT18} in the special case when the underlying graph is dense and regular (or close to regular). We also remark that regularity is a necessary assumption for for \cref{thm:mainthm}. Indeed, if $G$ is the complete bipartite graph with $n/3$ and $2n/3$ vertices on each side, respectively, then the local limit of the UST is a branching process (conditioned to survive) in which the progeny distribution alternates between Poisson$(1/2)$ and Poisson$(2)$, see \cite{HNT18}. \\ 



In the rest of this section we formally define the notion of local limits (\cref{sec:locallimit}), discuss the Poisson$(1)$ Galton-Watson tree conditioned to survive and show that \cref{thm:mainthm} follows from \cref{thm:mainthm2} (\cref{sec:poisson}), present the most general versions of our main theorems as described above (\cref{sec:extensions}) and close with a brief summary of our notation (\cref{sec:notationorganization}).

\subsection{Local limits}\label{sec:locallimit}

The local limit of finite graphs was first defined in the seminal paper of Benjamini and Schramm \cite{BeSc}, see also \cite{MR2354165} for further background.

We denote by $\Gb$ the space of all locally finite rooted graphs viewed up to root-preserving graph isomorphisms. In other words, the elements of $\Gb$ are pairs $(G,\rho)$ where $G$ is a graph and $\rho$ is a vertex of it, and two elements $(G_1,\rho_1), (G_2,\rho_2)$ are considered equivalent if there is a graph automorphism $\varphi:V(G_1)\to V(G_2)$ for which $\varphi(\rho_1)=\rho_2$. For a pair $(G,\rho)$ and an integer $r \geq 0$ we write $B_G(\rho,r)$ for the induced graph of $G$ on the vertex set which is of graph distance at most $r$ from $\rho$. 

We endow $\Gb$ with a natural metric: the distance between $(G_1,\rho_1)$ and $(G_2,\rho_2)$ is $2^{-R}$ where $R$ is the maximal integer such that there is a root-preserving isomorphism between $B_{G_1}(\rho_1,R)$ and $B_{G_2}(\rho_2,R)$; possibly $R=\infty$. We consider the Borel $\sigma$-algebra on this space and remark that with this metric $\Gb$ is a polish space, hence we can discuss convergence in distribution on it (see \cite{MR2354165}).

We say that a law of a random element $(G,\rho)$ of $\Gb$ is the {\bf local limit} of a sequence of (possibly random) finite graphs $G_n$ if for any integer $r\geq 0$ the random rooted graphs $B_{G_n}(\rho_n,r)$ converge in distribution to $B_{G}(\rho,r)$, where $\rho_n$ is a uniformly drawn random vertex of $G_n$ that is drawn independently conditioned on $G_n$.

\subsection{Poisson Galton-Watson trees} \label{sec:poisson}

Let us now describe the local limit of  \cref{thm:mainthm}. The Poisson$(1)$ Galton-Watson tree {\bf conditioned to survive} is defined as the local limit as $n\to \infty$ of a random rooted tree drawn according to a Poisson$(1)$ Galton-Watson tree conditioned to survive $n$ generations. 
It is classical (see \cite{conceptual95}) that this limit exists and can be drawn directly by taking an infinite path starting from the root vertex and hanging a Poisson$(1)$ unconditional Galton-Watson tree on each vertex of the path; however, we will not use this fact in this paper. \\


\noindent {\bf Proof of \cref{thm:mainthm} given \cref{thm:mainthm2}.} Let $(\T,\rho)$ be an instance of the Poisson$(1)$ Galton-Watson tree conditioned to survive and let $T$ be a finite rooted tree of height $r\geq 1$. Our goal is to prove that
\be\label{eq:poissongoal} \Prob ( B_\T(\rho,r)\cong T) = {|T_r| e^{-|V(T)|+|T_r|} \over |\stab_T|} \, .\ee

Denote by $\T^u$ an instance of an unconditional Poisson$(1)$ branching process. We will prove that 
\be\label{eq:poissongoal2} \Prob ( B_{\T^u}(\rho,r) \cong T ) = {e^{-|V(T)| + |T_r|}\over |\stab_T|} \, ,\ee
which implies \eqref{eq:poissongoal} by the following argument. Let $p_n$ denote the probability that $\T^u$ survives $n$ generations. It is well known (see \cite[Theorem C]{conceptual95}) that $p_n={2 +o(1) \over n}$. Thus, the probability that $B_{\T^u}(\rho,r) \cong T$ and $\T^u$ survives $n$ generations by \eqref{eq:poissongoal2} is just 
$$ {e^{-|V(T)| + |T_r|}\over |\stab_T|} \big ( 1 - (1-p_{n-r})^{|T_r|} \big ) = (1+o(1)) {e^{-|V(T)| + |T_r|}\over |\stab_T|}  |T_r| p_n \, .$$
Hence, if we condition on the event that $\T^u$ survives $n$ generations and take $n\to\infty$ we get precisely \eqref{eq:poissongoal}.

We prove \eqref{eq:poissongoal2} by induction on $r$. When $r=1$ we have that $\stab_T=k!$ where $k$ is the degree of the root of $T$ and \eqref{eq:poissongoal2} follows. Assume now that $r \geq 2$ and let $k$ be again the degree of the root of $T$ and call its children $\{u_1,\ldots, u_k\}$. For $1 \leq i \leq k$ write $T_{u_i}$ for the subtree of $T$ induced on all the descendants of $u_i$ (that is, all the vertices of $T$ for which the unique path to $u_i$ does not visit the root of $T$). We view $T_{u_i}$ as a tree rooted at $u_i$. 

We put an equivalence relation on $\{u_1,\ldots, u_k\}$ by declaring $u_i$ equivalent to $u_j$ if and only if there exists a root-preserving automorphism between $T_{u_i}$ and $T_{u_j}$. Let $\ell \leq k$ denote the number of equivalence classes and put $k_t$ for the number of elements in the $t$-th equivalence class, where $1 \leq t \leq \ell$, so that $k=k_1 + \ldots + k_t$. For $1 \leq t\leq \ell$ we also denote by $T^{(t)}$ the rooted tree $T_{u_i}$ where $u_i$ is a representative of the $t$-th equivalence class. Since $T$ is of height $r$ at least one of the $T^{(t)}$'s must be at height $r-1$ and all others have height at most $r-1$. A moment's reflection shows that 
\be\label{eq:stabidentity} |\stab_T| = \prod_{t=1}^{\ell} k_t! \prod_{i=1}^k |\stab_{T_{u_i}}| = \prod_{t=1}^{\ell}k_t! |\stab_{T^{(t)}}|^{k_t} \, .\ee

Now, for the event $B_{\T^u}(\rho,r) \cong T$ to occur, we must first have that the root of $\T^u$ has $k$ children (with probability $e^{-1}/k!$) and that there is a partition of the $k$ children of the root into $\ell$ subsets of sizes $\{k_t\}_{t=1}^\ell$ so that if a child of the root is at the $t$-th set, the $r-1$ generations of the branching process emanating from this child is isomorphic to $T^{(t)}$. Thus, by the induction hypothesis we have that
$$ \Prob (B_{\T^u}(\rho,r) \cong T ) = {e^{-1} \over k!} {k \choose k_1, k_2,\ldots, k_t} \prod_{t=1}^\ell \Big [ {e^{-|V(T^{(t)})|+ |T^{(t)}_{r-1}|} \over |\stab_{T^{(t)}}|} \Big ]^{k_t} \, .$$
Since $|V(T)|=1+\sum_{t=1}^\ell k_t |V(T^{(t)})|$ and $|T_r|=\sum_{t=1}^\ell k_t |T^{(t)}_{r-1}|$ and by \eqref{eq:stabidentity} we get that \eqref{eq:poissongoal2} holds, concluding our proof. \qed

\subsection{Extensions} \label{sec:extensions} The assumption of regularity in Theorems \ref{thm:mainthm}, \ref{thm:mainthm2} and \ref{thm:mainthm2quenched} can often be very restrictive. For instance, the removal of a single edge from the graph violates the assumptions though it is intuitively clear that the conclusion of the theorems will not be altered. Moreover, the proof of \cref{thm:mainthm2quenched}, even in the regular case, requires an extended version of \cref{thm:mainthm2} in which the graphs involved are ``close'' to being regular. For these two reasons we now state this extension.


\begin{defin}\label{def:almost} A sequence of finite, simple connected graphs $\{G_n\}$ is called {\bf high degree almost regular} if there exists some $d(n)\to\infty$ such that
\begin{enumerate}
	\item At least $(1-o(1))|V(G_n)|$ of the vertices of $G_n$ have degree $(1\pm o(1))d(n)$, and,
	\item The sum of degrees in $G_n$ is $(1\pm o(1))d(n)|V(G_n)|$.
\end{enumerate}
\end{defin}

It is immediately deduced from this definition that the total variation distance between a uniformly chosen vertex and a vertex drawn according to the stationary distribution is $o(1)$, indeed,
\be\label{eq:uniformstationary} \sum_{v} \Big | {\deg(v) \over \sum_v \deg(v)} - {1 \over n} \Big | = o(1) \, .\ee
We comment that it is straightforward to see that if the average degree of the graphs tends to infinity, then \eqref{eq:uniformstationary} in fact implies that the graph sequence is high degree almost regular. We will not use this direction in our proof though. 

\begin{thm}\label{thm:mainthm2Robust} Let $\{G_n\}$ be a sequence of high degree almost regular graphs. Then the conclusions of Theorems \ref{thm:mainthm} and \ref{thm:mainthm2} hold.
\end{thm}

Lastly, the following gives the quenched version of \cref{thm:mainthm2Robust}

\begin{thm}\label{thm:mainthm2quenched} Let $\{G_n\}$ be a sequence of high degree almost regular graphs. Let $\T_n$ be a UST of $G_n$ and $T$ be a fixed rooted tree of height $r\geq 1$. Denote by $Y_n(T)$ the number of vertices $v$ of $G_n$ satisfying $B_{\T_n}(v,r) \cong T$. Then with probability tending to $1$ as $n\to \infty$ we have that 
$$ Y_n(T) = (1+o(1)) \frac{|V(G_n)| |T_r| e^{-|V(T)|+|T_r|}}{|\stab_T|} \, .$$
\end{thm}


\subsection{Notation} \label{sec:notationorganization}
\begin{itemize}
\item Given a graph $G$ we write $V(G)$ and $E(G)$ for its vertex and edge set, respectively. 
\item For an integer $k\geq 1$ we write $[k]$ for $\{1,\ldots, k\}$. 
\item We write $\deg(u)$ for the degree of a vertex $u$.
\item For vertices $u,v\in V(G)$ we write $u\sim v$ if there is an edge between them.
\item For two positive sequences $\{a_n\}, \{b_n\}$ we write $a_n = O(b_n)$ when there exists a constant $C$ such that $a_n \leq C b_n$ for all $n$, and $a_n = o(b_n)$ when $a_n/b_n$ tends to $0$.
\end{itemize}





\section{Preliminaries}
\subsection{Random walks and electric networks} We give here a brief review of the theory of electric networks and refer the reader to \cite[Chapter 2]{LyPe:ProbabilityTrees} or \cite[Chapter 2]{NStFlour18} for more details. 

A {\bf network} is a connected graph $G$ equipped with positive edge weights $c(u,v)$ for any edge $(u,v)\in E(G)$. We write $\pi(v) = \sum_{u : u \sim v} c(u,v)$, where possibly $u=v$ if there is a loop at $v$. The {\bf network random walk} on the network is a reversible Markov chain $\{X_t\}_{t=0}^\infty$ on $V(G)$ with transition probabilities 
$$ p(u,v) = {c(u,v) \over \pi(v)} \, .$$
We write $\Prob_u(\cdot)$ for the probability measure of the chain conditioned on $X_0=u$ and $\E_u$ for the corresponding expectation. We denote by $\tau_u$ and $\tau^+_u$ the stopping times
$$ \tau_u = \min \{ t \geq 0 : X_t = u\} \qquad \tau^+_u = \min \{ t \geq 1 : X_t = u\} \, ,$$
which are of course equal unless $X_0=u$. 

The \emph{effective resistance} between two vertices $u,v$ of the network is the voltage difference between $u$ and $v$ when unit current flows from $u$ to $v$. We will not use this definition directly but rather a probabilistic interpretation of it which can serve as its definition. We define the {\bf effective resistance} between $u$ and $v$ in the network $G$ by
\be\label{def:reff} \Reff(u \lr v ; G) = {1 \over \pi(u) \Prob_u(\tau_v < \tau^+_u) } \, .\ee
We often write $\Reff(u \lr v)$ when the underlying network is evident from the context.

\subsubsection{Nash-Williams inequality} The following is useful method of bounding the resistance from below and is due to Nash-Williams \cite{Nash-Williams}. We say that a subset of edges of the network {\bf separates} the vertices $a$ and $z$ if every path from $a$ to $z$ uses at least one edge of the subset. If $\Pi_1, \ldots, \Pi_n$ are disjoint subsets of edges separating $a$ and $z$, then 
$$ \Reff(a \lr z) \geq \sum_{i=1}^n \Big ( \sum_{e \in \Pi_i} c_e \Big )^{-1} \, ,$$
see \cite[Chapter 2.5]{LyPe:ProbabilityTrees}. An immediate application of this inequality is that 
\be\label{eq:reslowerbound} \Reff(u \lr v) \geq {1 \over \deg(u) + 1} + {1 \over \deg(v)+1} \, ,\ee
whenever $u\neq v$ are vertices of a simple graph and all the conductances are unit. Indeed, if $u$ and $v$ do not share an edge the lower bound can be improved by dropping the two $+1$'s in the denominators since the edges emanating from $u$ separate $u$ from $v$ and similarly for the edge emanating from $v$. If $u$ and $v$ do share an edge, we replace that edge with a path of length $2$ and assign conductance $2$ to the two edges of the path and apply again Nash-Williams inequality. 

\subsubsection{The commute-time identity} The commute time between two vertices $a\neq z$ in a network is $\E_a \tau_z + \E_z \tau_a$. In other words, it is the expected time it takes to the random walk started at $a$ to visit $z$ and then go back to $a$. It turns out \cite{commute} that this quantity is just the rescaled resistance between $a$ and $z$ (see also \cite[Corollary 2.21]{LyPe:ProbabilityTrees}):
\be\label{eq:commute} \E_a \tau_z + \E_z \tau_a = 2|E(G)| \Reff(a \lr z) \, ,\ee
in any network $G$ with unit conductances.

\subsection{Uniform spanning trees} Let $G$ be a finite connected graph. The set of spanning trees of $G$ is non-empty and finite hence we may draw one uniformly at random; denote by $\T$ the resulting spanning tree. We call $\T$ a {\bf uniform spanning tree} (UST) of $G$. In this paper we will only use two basic properties of the UST.

\subsubsection{Kirchhoff's formula} Kirchhoff \cite{Kirchhoff} discovered a fundamental relation between the uniform spanning tree and electric networks. Let $e=(x,y)\in E(G)$. Then
\be\label{eq:kirchhoff} \Prob ( e \in \T ) = \Reff(x \lr y; G) \, .\ee

\subsubsection{Spatial Markov property of the UST} Let $A,B\subset E(G)$ be two disjoint subset of edges of $G$. We would like to condition on the event $A \subset \T$ and $B \cap \T = \emptyset$. For this event to have positive probability we must have that $A$ does not contain a cycle and that when erasing the edges of $B$ from $G$ we remain with a connected graph. It turns out that this conditional measure can be described as drawing a UST on a modified graph, see \cite[Chapter 4]{LyPe:ProbabilityTrees}. We denote by $G/A-B$ the graph obtained from $G$ by contracting the edges of $A$ and erasing the edges of $B$. 
\begin{prop} \label{prop:spatialMarkov} Let $G$ be a finite connected graph and $A,B$ two disjoint subsets of edges such that $G-B$ is connected and $A$ has no cycles. Then the law of the UST on $G$ conditioned on the event $A \subset \T$ and $B\cap T=\emptyset$ equals the law of the UST on $G/A-B$, viewed as a random edge subset of $G$, union the edges of $A$.
\end{prop}

Lastly we recall the negative correlations of the UST. It is an immediate corollary of \eqref{eq:kirchhoff} and \cref{prop:spatialMarkov} together with Rayleigh's monotonicity for electric networks \cite[Chapter 2]{LyPe:ProbabilityTrees} that the UST has negative correlations. That is, if $e\neq f$ are two edges of a finite connected graph and $\T$ is a UST of $G$, then
\be\label{eq:negativecorr} \Prob( e \in \T \and f \in \T ) \leq \Prob(e\in \T) \Prob( f \in \T) \, ,\ee
see \cite[Chapter 4]{LyPe:ProbabilityTrees}.

\section{Foster's Theorem and variants}\label{sec:Foster}

Recall that Foster's Theorem \cite{Foster48} (which also follows from \eqref{eq:kirchhoff} immediately) asserts that on any connected graph $G$ 
\be\label{thm:foster} \sum_{\{x,y\}\in E(G)} \Reff(x \lr y) = |V(G)|-1 \, .\ee
When $G$ is a regular graph with degree $d$, its number of edges is ${|V(G)|d \over 2}$. Hence Foster's Theorem in this case states that 
\be\label{eq:fostereff} \Exp \Reff(X_0 \lr X_1) = {2  \over d} - {2 \over |V(G)|d} \, ,\ee
where $\{X_0,X_1\}$ is a uniformly drawn edge of $G$, or equivalently, $X_0$ is a uniformly drawn vertex of $G$ and $X_1$ is an independent uniform neighbor of $X_0$. By \eqref{eq:reslowerbound} the resistance between two vertices in a $d$-regular graph is deterministically lower bounded by ${2 \over d+1}$. Since this bound is rather close to the expectation above, it gives a useful concentration estimate for the random variable $\Reff(X_0 \lr X_1)$.

\begin{lemma} \label{lem:highrescount} On any $d$-regular graph $G$ and for any $\e>2/d$ the number of edges $e=(x,y)$ with $\Reff(x \lr y) \geq \e$ is at most ${|V(G)| \over \e d - 2}$.
\end{lemma}
\begin{proof}
Denote by $N_\e$ the number of such edges. By \eqref{eq:reslowerbound} we bound $\Reff(x\lr y) \geq {2 \over d+1} \geq {2 \over d} - {2 \over d^2}$. Hence by Foster's Theorem \eqref{thm:foster}
$$ |V(G)|-1 \geq N_\e \e + \Big ( {|V(G)|d \over 2} - N_\e \Big ) (2/d - 2/d^2) \, ,$$
giving the required upper bound on $N_\e$. \qedhere
\end{proof}

Our first variant is an estimate of the resistance between two endpoints of a longer random walk.

\begin{lemma}\label{lem:fosterRW} Let $(X_0,\ldots,X_k)$ be a $k$-step random walk, $k\geq 1$, on a simple $d$-regular graph $G$ starting from a uniformly drawn vertex $X_0$. Then
$$ \Exp \Reff(X_0 \lr X_k) \leq {2 \over d} + {2(k-1) \over d^2}  \, .$$
\end{lemma}

\noindent{\bf Remark.} In fact the proof gives that 
$$ \Exp \Reff(X_0 \lr X_k) \leq {2 \over d} + {2(k-1) \over d^2} - {2 k \over |V(G)|d} \big ( 1-{1 \over d} \big )^{k-1} - {2k \over |V(G)|d(d-1)} \sum_{i=1}^{k-1} \big ( 1 - {1 \over d} \big )^{i}  \, ,$$
which is an equality when $k=1$ by \eqref{eq:fostereff}. 

\begin{proof} In any irreducible finite Markov chain the expected return time from a state to itself equals the inverse of its stationary mass. Thus, since $G$ is regular we have that $\E_{x_0}\tau^+_{x_0} = |V(G)|$ for any vertex $x_0$ of $G$. Hence for any vertex $x_0$ and $k\geq 1$
\begin{eqnarray}\label{eq:x1return} |V(G)| \geq \E_{x_0} \tau^+_{x_0} {\bf 1} _{\{\tau^+_{x_0} > k\}} &\geq& {1 \over d^{k}} \sum_{\substack{(x_1,\ldots,x_k) \\ x_i \neq x_0 \forall i=1,\ldots,k}} \big [ \E_{x_k} \tau_{x_0} + k \big ]
\, ,\end{eqnarray}
where $(x_1,\ldots,x_k)$ is a walk of length $k$ in $G$ starting from a neighbor $x_1$ of $x_0$.
On the other hand we have that
$$ \E \E _{X_k} \tau_{x_0} = {1 \over d^{k} } \sum_{\substack{(x_1, \ldots, x_k)}} \E_{x_k} \tau_{x_0} \, .$$
The last sum contains all the terms on the right hand side of \eqref{eq:x1return} except those for which there is an $i\in[k-1]$ for which $x_{i}=x_0$ (if $x_k=x_0$ the hitting time is $0$ so the case $i=k$ is dismissed). We let $i$ be the last such index and by \eqref{eq:x1return} we get 
\be\label{eq:xkhitx1} \E \E _{X_k} \tau_{x_0} \leq |V(G)| - k \big (1 - {1 \over d} \big )^{k-1} + {1 \over d^{k}} \sum_{i=1}^{k-1} \sum_{\substack{(x_1,\ldots,x_{i}) \\ x_i=x_0}} \sum_{\substack{(x_{i+1}, \ldots, x_k) \\ x_j \neq x_0 \forall j=i+1,\ldots,k}} \E_{x_k} \tau_{x_0} \, ,\ee
where for the second term on the right hand side we bounded the number of $(x_1,\ldots,x_k)$ in \eqref{eq:x1return} below by $d(d-1)^{k-1}$. By \eqref{eq:x1return} we have that for each $i\in[k-1]$, if $x_i=x_0$, then
$$ \sum_{\substack{(x_{i+1}, \ldots, x_k) \\ x_j \neq x_0 \forall j=i+1,\ldots,k}} \E_{x_k} \tau_{x_0} \leq |V(G)| d^{k-i} - kd(d-1)^{k-i-1}\, .$$
We put this back into \eqref{eq:xkhitx1} and bound the number of $(x_1,\ldots,x_{i})$ for which $x_i=x_0$ by $d^{i-1}$, sum over $i$ and obtain
$$ \E \E _{X_k} \tau_{X_0} \leq |V(G)|\big ( 1+ {k-1 \over d}\big ) - k \big ( 1-{1 \over d} \big )^{k-1} - {k \over d-1} \sum_{i=1}^{k-1} \big ( 1 - {1 \over d} \big )^{i} \, .$$
Since $G$ is $d$-regular and $X_0$ is a uniformly drawn vertex we have that $(X_0,\ldots,X_k)$ has the distribution as $(X_k,\ldots, X_0)$, hence
$$ \E \big [ \E _{X_0} \tau_{X_k} + \E _{X_k} \tau_{X_0} \big ] \leq 2|V(G)|\big ( 1+ {k-1 \over d}\big ) - 2k \big ( 1-{1 \over d} \big )^{k-1} - {2k \over d-1} \sum_{i=1}^{k-1} \big ( 1 - {1 \over d} \big )^{i}\, .$$
By the commute-time identity \eqref{eq:commute} we get that
$$ \E \Reff (X_0 \lr X_k) \leq {2 \over d} + {2(k-1) \over d^2} - {2 k \over |V(G)|d} \big ( 1-{1 \over d} \big )^{k-1} - {2k \over |V(G)|d(d-1)} \sum_{i=1}^{k-1} \big ( 1 - {1 \over d} \big )^{i}  \, ,$$
concluding our proof.
\end{proof}

As in \cref{lem:highrescount} we obtain a concentration estimate using \eqref{eq:reslowerbound}. 

\begin{cor}\label{cor:highrescountRW} Let $(X_0,\ldots,X_k)$ be a $k$-step random walk on a simple $d$-regular graph $G$ starting from a uniformly drawn vertex $X_0$. Then for any $\e > {2 \over d}$ we have that
$$ \Prob \big ( \Reff(X_0 \lr X_k) \geq \e \big ) \leq {2k \over \e d^2 - 2d } \, .$$
\end{cor}
\begin{proof} Denote the probability above by $p$. By \eqref{eq:reslowerbound} the resistance between distinct vertices is always at least ${2 \over d} - {2 \over d^2}$. By \cref{lem:fosterRW} 
$$ {2 \over d} + {2k -2 \over d^2} \geq \Exp \Reff(X_0 \lr X_k) \geq \e p + (1-p)\Big ( {2 \over d} - {2 \over d^2} \Big ) \, ,$$
and rearranging gives the result. \end{proof}

We present now our final variant of Foster's Theorem. Let $T$ be a fixed rooted tree with $k\geq 3$ vertices and $G$ a $d$-regular graph with $n$ vertices; we think of $n$ and $d$ as large and $k$ as fixed. We denote $T$'s vertex set by $[k]$ so that $1$ is the root and such that for any $2 \leq i \leq k$ the graph spanned on $[i]$ is a tree and the vertex $i$ is a leaf. For instance, we can label all vertices according to a breadth-first search (BFS) order. 


We say that a $k$-tuple of vertices $(v_1,\ldots,v_k)$ of $G$ is {\bf $T$-compatible} when they are distinct vertices and for any $i,j\in[k]$ such that the pair $\{i,j\}$ is an edge of $T$, the pair $\{v_i,v_j\}$ is an edge of $G$. We draw a random $k$-tuple in the following iterative fashion. Let $X_1$ be a uniformly drawn vertex of $G$; for $2 \leq i \leq k-1$ assume we have drawn $(X_1,\ldots, X_{i-1})$ and let $1 \leq j \leq i-1$ be the unique index such that the vertex $i$ of the subtree of $T$ induced on $[i]$ is a leaf hanging on $j$. We draw $X_i$ to be a random neighbor of $X_j$. When $T$ is a path with $k$ vertices, the tuple $(X_1,\ldots,X_k)$ is distributed just as $k-1$ steps of the simple random walk starting from a uniform vertex. Let us remark that the tuple $(X_1,\ldots,X_k)$ drawn this way is not necessarily $T$-compatible since the vertices are not necessarily distinct, but with high probability it will be $T$-compatible.

\begin{thm} \label{thm:GeneralFoster} Let $G$ be a simple $d$-regular graph and $T$ a rooted tree on $k\geq 3$ vertices such that $d \geq 16 k \log ^k d$. Let $\{X_1,\ldots,X_k\}$ be a random $k$-tuple drawn as described above. Then with probability at least $1 - {2k^3 \over \log^k d}$ we have that $(X_1,\ldots,X_k)$ are $T$-compatible and 
$$ \Big | \Reff \big (X_k \lr \{X_1, \ldots, X_{k-1}\} \big ) - {k \over (k-1) d} \Big | \leq {72 k \log^k d \over d^2} \, .$$
\end{thm}

We first prove the following lemma.

\begin{lemma} \label{thm:reff1tomany} Let $G$ be a network and $[k]$ are distinct vertices of it with $k \geq 3$. Let $x>0$ be given and $\e>0$ satisfy $\e \leq {x \over 32k}$. Assume that for any distinct $i \neq j$ in $[k]$ one has
\be\label{eq:reff1assumption} | \Reff ( i \lr j ) -x | \leq \e \, .\ee
Then 
\be\label{eq:reff1conclusion} \Big | \Reff\big (1 \lr \{2,\ldots, k\}\big ) - {kx\over 2(k-1)} \Big | \leq 72 k \e  \, . \ee
\end{lemma}
\begin{proof} Let $p(i,j)$ be the probability that the network random walk on $G$ started at $i$ is at $j$ when it first returns to $[k]$; possibly $i=j$. Then $p$ is the transition matrix for the network random walk on $[k]$ with conductances $c(i,j) = \pi(i)p(i,j)$ ($=c(j,i)$ by reversibility). This implies that $\pi(i)$ remains unchanged in this reduced network and by \eqref{def:reff} we conclude that the pairwise effective resistances as well as the effective resistance between $1$ and $\{2,\ldots,k\}$ also remain unchanged. Hence \eqref{eq:reff1assumption} holds and it suffices to prove \eqref{eq:reff1conclusion} for this reduced chain on $[k]$.  Since loops do not change the resistance we remove them from this network, and we denote $\tilde{\pi}(i) = \sum_{j \neq i} c(i,j)$ the modified stationary measure after removing the loops.

Let $\Delta$ be the {\bf network Laplacian}, that is, $\Delta$ is a $k\times k$ symmetric matrix defined by $\Delta(i,i)=\tilde{\pi}(i)$ for any $i \in [k]$ and $\Delta(i,j) = -c(i,j)$ for any distinct $i\neq j$ in $[k]$. Given $a\in[k]$ we write $\Delta[a]$ for the $(k-1)\times (k-1)$ matrix obtained from $\Delta$ by erasing the row and column indexed by $a$. 

For three vertices $a,i,j \in [k]$ with $i,j\in[k] \setminus \{a\}$, we write $g_a(i,j)=G_{a}(i,j)/\tilde{\pi}(j)$ for the normalized Green's function of the random walk killed at $a$, that is, $G_{a}(i,j)$ is the expected number of visits to $j$ of a random walk started at $i$ before it visits $a$. We view $g_a(\cdot,\cdot)$ as a $(k-1)\times(k-1)$ matrix. The quantity $g_a(i,j)$ is also the voltage at $i$ when unit current flows from $a$ to $j$ (so that the voltage at $a$ is zero). Hence, $g_a(i,j) = \Prob_i(\tau_j < \tau_a) \cdot \Reff(a \lr j)$. It is well known that this quantity can be expressed in terms of the pairwise effective resistances, indeed, by Exercise 2.68 of \cite{LyPe:ProbabilityTrees} we get that 
$$ g_a(i,j) = {\Reff(a \lr j) + \Reff(a\lr i) - \Reff(i \lr j) \over 2} \, ,$$
from which we deduce by our assumption \eqref{eq:reff1assumption} that $|g_a(i,j)-x/2| \leq 3\e/2$ for all $i \neq j$ in $[k]\setminus \{a\}$ and $|g_a(i,i) -x|\leq \e$ for all $i\in [k]\setminus \{a\}$. 

It is classical (see Exercise 2.62(a) in \cite{LyPe:ProbabilityTrees}) that $g_a(\cdot,\cdot)$ is an invertible matrix and that $\Delta[a] = \big [g_a(\cdot,\cdot) \big ]^{-1}$. Denote by $A$ the $(k-1)\times(k-1)$ matrix with $x$ on the diagonal and $x/2$ in all other entries, so that $g_a = A+E$ and $||E||_1 \leq {3\e \over 2}(k-4/3)$ where $||E||_1$ the maximum $\ell_1$ norm of the rows of $E$ viewed as vectors in $\mathbb{R}^{k-1}$.


It is straightforward to verify that $A^{-1}$ is a matrix with $\alpha$ on the diagonal and $\beta$ in all other entries, where
$$ \alpha = {2(k-1) \over kx} \qquad \qquad \beta = -{2 \over kx} \, ,$$
so that $||A^{-1}||_1 = {4k-6 \over kx}$. It is a well known fact \cite{inversestable} that if $||A^{-1}||_1\cdot||E||_1<1$, then 
$$ ||(A+E)^{-1} - A^{-1} ||_1 \leq {||A^{-1}||_1^2 ||E||_1\over 1-||A^{-1}||_1||E||_1} \, .$$
Since $\e \leq {x \over 32 k}$ we learn that $||A^{-1}||_1\cdot||E||_1\leq 1/4$
from which we conclude that $||\Delta[a] - A^{-1}||_1 \leq {32 k \e / x^{2}}$. Since $a$ was arbitrary for any $i,j\in [k]$ with $i\neq j$ we get
$$ \Big |c(i,j) - {2 \over kx} \Big | \leq {32 k \e \over x^{2}} \qquad \qquad \Big | \tilde{\pi}(i) - {2(k-1) \over kx} \Big | \leq {32 k \e \over x^{2}} \leq {1 \over x}\, ,$$
since $\e \leq {x \over 32 k}$. Since $k \geq 3$ we have that $\tilde{\pi}(i) \geq {1 \over 3x}$ and ${k \over k-1} \leq 3/2$. Thus, 
$$\Big | {1 \over \tilde{\pi}(i)} - {kx\over 2(k-1)} \Big | \leq {32 k^2 \e/x \over 2(k-1)\tilde{\pi}(i)} \leq  {72 k \e} \, ,$$
where we used that if $|a-b|\leq c$ then $|{1 \over a}-{1 \over b}| \leq c/|ab|$.

Since there are no loops in our reduced network, the network random walk started at $1$ visit $\{2,\ldots,k\}$ before returning to $1$ with probability $1$. Hence $\Reff(1 \lr \{2,\ldots,k\})$ is just $1 \over \tilde{\pi}(1)$ by \eqref{def:reff}, concluding the proof.
\end{proof}

\noindent{\bf Proof of \cref{thm:GeneralFoster}.} Put $\e = {2 + {\log^k d \over d} \over d}$. For any $i,j\in[k]$ with $i\neq j$ the random vertex $X_i$ is uniformly distributed over the vertices of $G$ and $X_j$ is the endpoint of a random walk starting at $X_i$ of length $\ell \leq k$, where $\ell$ is the length of the path between $i$ and $j$ in $T$. By \cref{cor:highrescountRW} 
$$ \Prob \Big ( \Reff(X_i \lr X_j) \geq {2 + {\log^k d\over d} \over d} \Big ) \leq {2k \over \log^k d } \, .$$
By the union bound
$$\Prob \Big ( \exists i\neq j \in [k] \quad \Reff(X_i \lr X_j) \geq {2 + {\log^k d\over d} \over d} \Big ) \leq {k^3 \over \log^k d } \, .$$
Furthermore, by \eqref{eq:reslowerbound}, if $X_i \neq X_j$, then the resistance across them is at least ${2 \over d} - {2 \over d^2}$. At each step of the random process of drawing $(X_1,\ldots,X_k)$, the probability we draw an already chosen vertex is at most $k/d$ hence the probability that $(X_1,\ldots,X_k)$ is compatible with $T$ (in particular they are distinct vertices) is at least $1-{k^2 \over d} \geq 1-{k^3 \over \log^k d }$ by our assumption on $d$. 

We conclude that with probability at least $1-2k^3/\log^k d$ for all $i,j\in[k]$ with $i\neq j$ we have
$$ \Big | \Reff(X_i \lr X_j) - {2 \over d} \Big | \leq {\log ^k d \over d^2} \, ,$$
from which \cref{thm:reff1tomany} implies that desired result (we have used our assumption $d \geq 16 k \log ^k d$ to verify the assumption of \cref{thm:reff1tomany}).\qed

\begin{cor} \label{cor:generalhighrescount} Let $G$ be a simple $d$-regular graph and $T$ a rooted tree on $k\geq 3$ vertices such that $d \geq 16 k \log ^k d$. Denote by $N$ the number of $k$-tuples $(v_1,\ldots,v_k)$ that are $T$-compatible and such that 
$$ \Big | \Reff ( v_k \lr \{v_1,\ldots,v_{k-1}\}) - {k \over (k-1) d} \Big | \geq {72 k \log^k d \over d^2} \, .$$
Then
$$ N \leq {2 nd^{k-1}k^3 \over \log^k d}\, .$$ 
\end{cor}
\begin{proof}
Immediate from \cref{thm:GeneralFoster}.
\end{proof}






\section{Tightness}\label{sec:tightness}


Our goal in this section is to prove \cref{thm:tightness} showing that, in the setting of \cref{thm:mainthm}, for any integer $r\geq 1$ the random variables $|B_{\T_n}(X,r)|$ are a tight sequence. We remark that for $r=1$ this is trivial since the expected degree of a randomly chosen vertex in {\em any} tree at most $2$. This reasoning fails for $r \geq 2$, indeed, the behavior of bigger balls is inherently different, for example, $\Exp |B_{\Gamma_n}(X,3)|$ may be unbounded, see \cref{subsec:questionsballs} for an example and further discussion.

\begin{lemma} \label{lem:stayawayfrombad} Let $\T$ be a fixed finite tree, $S\subset V(\T)$ be a subset of its vertices and $X$ a uniformly drawn vertex of $V(\T)$. Then for any positive integers $r, m, M$ we have
$$ \Prob \Big ( B_\T(X,r-1) \cap S \neq \emptyset \, ,\,  |B_\T(X,r-1)| \leq m \, , \,  |B_\T(X,r)| \leq M  \Big ) \leq \frac{r|S|m^{r-2} M} {|V(\T)|} \, .$$
\end{lemma}


\begin{proof} If $x\in \T$ is a vertex such that $B_\T(x,r-1) \cap S \neq \emptyset$, then there exists a vertex $v \in S$ and a simple path in $\T$ of length $\ell \leq r-1$ from $x$ to $v$. If in addition $|B_\T(x,r-1)| \leq m$ and $|B_\T(x,r)| \leq M$, then when $\ell<r-1$ all vertex degrees on this path are at most $m$ and if $\ell=r-1$, then the degrees of the first $r-2$ vertices of the path are at most $m$ and the degree of $v$ is at most $M$. For each $v \in S$, the number of such paths of length $\ell<r-1$ is at most $m^{\ell}$ and the number of such paths when $\ell = r-1$ is at most $M m^{r-2}$. Hence the number of possible $x$'s is at most 
$$ |S| \sum_{\ell=1} ^{r-2} m^{\ell} + |S| m^{r-2} M \leq r|S|m^{r-2} M \, ,$$
concluding the proof.
\end{proof}

\begin{thm} [Tightness] \label{thm:tightness} Let $\{G_n\}$ be a sequence of finite, simple, connected, regular graphs with degree $d(n)\to \infty$. Let $\T_n$ be a uniformly drawn spanning tree of $G_n$ and let $X$ be a uniformly chosen random vertex of $G_n$. Then for any integer $r\geq 0$ we have
$$ \lim _{M \to \infty} \sup_n \Prob \Big ( |B_{\T_n}(X,r)| \geq M \Big ) = 0 \, .$$
\end{thm}
\begin{proof} To simplify the notation we write $B_{r}$ for the random variable $|B_{\T_n}(X,r)|$. Our proof is by induction. The case $r=0$ is obvious since $B_0=1$ always. Let $r \geq 1$ and let $\e>0$ be arbitrary. By induction we know that there exists some $m>0$ such that for all $n$ we have
\be \label{eq:tightstep1} \Prob \Big ( B_{r-1} \geq m \Big ) \leq \e/2 \, .\ee
We fix $M \geq m$ that we will choose later depending only on $m,r$ and $\e$. Our goal is to estimate $ \Prob ( B_r \geq M \and B_{r-1} \leq m )$. This equals
\be\label{eq:tightsplitk} \sum_{k=0}^\infty \Prob \Big ( B_r \in [2^k M, 2^{k+1}M) \and B_{r-1} \leq m \Big ) \, .\ee
For each $k\geq 0$ we choose some large $A_k \geq 4$ that will also be chosen later (it will depend only $k, r, m$ and $\e$) and apply \cref{lem:highrescount} to obtain that the number of edges with resistance on them at least $A_k/d$ is at most $n/(A_k-2) \leq 2n/A_k$. Let $S_k$ denote the set of vertices which touch at least $2^{k-1} M/m$ edges with resistance at least $A_k/d$. Thus, $|S_k| \leq \frac{8n m}{A_k 2^{k} M}$. By \cref{lem:stayawayfrombad} we have 
\be \label{eq:tightstep2} \Prob ( B_r \in [2^k M, 2^{k+1}M) \and B_{r-1} \leq m \and B_{\T_n}(X,r-1) \cap S_k \neq \emptyset \Big ) \leq \frac{16 r m^{r-1} }{A_k} \, .\ee
We now upper bound
\be\label{eq:tightstep4} \Prob \Big ( B_r \in [2^k M, 2^{k+1}M) \and B_{r-1} \leq m \and B_{\T_n}(X,r-1) \cap S_k = \emptyset \Big ) \, .\ee
If the event above occurs, then there is a path $v_0,v_1,\ldots,v_{r-1}$ in $\T_n$ starting at $v_0=X$ of length $r-1$ such that 
\begin{enumerate}
	\item All vertices in the path are not in $S_k$, and,
	\item $v_{r-1}$ has at least $D=2^kM/m$ edges of $\T_n$ touching it whose other endpoint is not in $B_{\T_n}(X,r-1)$.  
%
\end{enumerate}
Since $v_{r-1} \not \in S_k$ at least $D/2$ of the edges specified above have resistance at most $A_k/d$ between their endpoints. We enumerate over all the possibilities of the path $(v_0,\ldots,v_{r-1})$ (there are at most $nd^{r-1}$ choices) and $D/2$ children of $v_{r-1}$ (there are ${d \choose D/2}$ choices). We apply \eqref{eq:negativecorr} and bound the probability of each edge being in $\T_n$ by $A_k/d$. We obtain that \eqref{eq:tightstep4} is bounded by

$$ d^{r-1} (A_k/d)^{r-1} {d \choose D/2} (A_k/d)^{D/2} \leq A_k^{r-1} \Big ( {2eA_k \over D} \Big )^{D/2} \, .$$




We put this bound together with \eqref{eq:tightstep1}, \eqref{eq:tightsplitk} and \eqref{eq:tightstep2} to obtain that
\be\label{eq:lastineq} \Prob ( B_r \geq M ) \leq {\e \over 2} + \sum_{k=0}^\infty {16 r m^{r-1} \over A_k} + \sum_{k=0}^\infty A_k^{r-1} \Big ( {2eA_k \over D} \Big )^{D/2}  \, ,\ee
where $D=2^k M/m$. We now choose $A_k = 2^{k+7} r m^{r-1} \e^{-1} $ so that the first sum on the right hand side is at most $\e/4$. We then take $M=M(m,r,\e)$ so large such that for any $k\geq 0$ we have
$$ A_{k}^{r-1} \Big ( {2eA_k \over D} \Big )^{D/2} \leq \e(1/4)^k /8 \, ,$$
for example, $M$ can be chosen so that $2eA_k/D \leq \e^{2r} 2^{-r} r^{-r}m^{-r^2}$ and $D/2 \geq 2^{k+1}$ so that the third sum in \eqref{eq:lastineq} is at most $\e/4$. We get that $\Prob ( B_r \geq M ) \leq \e$, concluding the proof. \end{proof}

\section{Proof of main theorem}


We say that a vertex of $G$ is {\bf good} if it does not touch an edge such that the effective resistance between its endpoints is at least ${\log d \over d}$. Given a fixed rooted tree $T$ of height $r$ on $k\geq 2$ vertices, we say that $T$-compatible $k$-tuple $(v_1,\ldots,v_k)$ is ${\bf good}$ if all the vertices of tree-distance at most $r-1$ to the root are good. In other words, if the vertices $(v_1,\ldots, v_{k-|T_r|})$ are good vertices. The following is a key calculation.

\begin{lemma}\label{lem:openedges} Let $G$ be a simple $d$-regular graph and $T$ a rooted tree on $k\geq 2$ vertices. If $d \geq 72 k \log ^{k+1} d$, then
\be\label{eq:openedges} {1 \over n} \sum_{\substack{(v_1,\ldots, v_k) \\ T\mathrm{-compatible} \\ \mathrm{good}}} \prod_{i=2}^k \Reff(v_i \lr \{v_1,\ldots, v_{i-1}\}) \leq k + {4^k 2 k^3  \over \log d} \, ,\ee
\end{lemma}
\begin{proof} We prove by induction on $k$. 
The base case $k=2$ follows (without the second term on the right hand side of \eqref{eq:openedges}) by Foster's Theorem \eqref{thm:foster}. We proceed to the general $k\geq 3$ case.

We write $C=2k^3$. By \cref{cor:generalhighrescount} the number of $T$-compatible $k$-tuples $(v_1,\ldots,v_k)$ for which 
\be\label{eq:resineq} \Reff(v_k \lr \{v_1,\ldots,v_{k-1}\}) \geq {k + {72 k\log^{k} d \over d} \over (k-1)d}\ee
is at most $Cnd^{k-1} / \log^{k} d$. For these tuples, we bound each term in the product by ${\log d \over d}$; indeed, we may do so since $v_i$ has a good neighbor in $\{v_1,\ldots,v_{i-1}\}$ for all $i\in \{2,\ldots,k\}$. This gives us an upper bound of 
$$ {1 \over n} \times {Cnd^{k-1} \over \log^{k} d} \times \Big ( {\log d \over d} \Big )^{k-1} = {C \over \log d} \, .$$
For all other tuples we have the opposite inequality at \eqref{eq:resineq}. Thus we may bound the last term in the product by this, sum it over the $d$ possible choices of $v_k$ and obtain that the sum at the left hand side of \eqref{eq:openedges} is bounded above by

\begin{eqnarray*} {C \over \log d} + {k + {1 \over \log d} \over (k-1)} \cdot {1 \over n} \cdot \sum_{\substack{(v_1,\ldots, v_{k-1}) \\ T\setminus\{v_k\}\mathrm{-compatible} \\ \mathrm{good}}} \prod_{i=2}^{k-1} \Reff(v_i \lr \{v_1,\ldots, v_{i-1}\}) \, ,\end{eqnarray*}
where we used our assumption on $d$ to bound ${72 k \log^{k} d \over d} \leq {1 \over \log d}$. We use our induction hypothesis to bound the last term and get that
\begin{eqnarray*} {1 \over n} \sum_{\substack{(v_1,\ldots, v_k) \\ T\mathrm{-compatible} \\ \mathrm{good}}} \prod_{i=2}^k \Reff(v_i \lr \{v_1,\ldots, v_{i-1}\}) &\leq& {C \over \log d} + {k + {1 \over \log d} \over (k-1)} \Big ( k-1 + {4^{k-1}C \over \log d} \Big ) \\ &\leq& k + {2C \over \log d} + {3 \cdot 4^{k-1} C \over \log d} \, ,\end{eqnarray*}
where we bounded ${k + {1 \over \log d} \over (k-1)} \leq 3$. Since $2+3\cdot 4^{k-1}\leq 4^k$ we get the desired inequality and conclude our proof.
\end{proof}

\noindent {\bf Proof of \cref{thm:mainthm2}.}
Let $(v_1,\ldots,v_k)$ be a $T$-compatible $k$-tuple of $G_n$. We write $T(v_1,\ldots,v_k)$ for the subset of edges $\big \{ \{v_i,v_j\} : \{i,j\} \in E(T) \}$ of $G_n$ (by definition of $T$-compatible all of these pairs must be edges of $G_n$). We have that 
$$ \Prob (B_{\T_n}(X, r) \cong T) = {1 \over |V(G_n)| |\stab_T|} \sum_{\substack{(v_1,\ldots, v_k) \\ T\textrm{-compatible}}} \Prob ( B_{\T_n}(v_1, r) = T(v_1,\ldots,v_k) ) \, ,$$
since for each root preserving isomorphism $\sigma\in \stab_T$ the edge subsets $T(v_1,\ldots,v_k)$ and $T(v_{\sigma(1)},\ldots,v_{\sigma(k)})$ are equal so the corresponding events identify as well. Up to these isomorphisms, all the events are disjoint and hence the $|\stab_T|^{-1}$ term above.

Recall that $T_r$ are the vertices at the last level of $T$ viewed from the root. Set $t=k-|T_r|$ so that by our labeling convention the vertices $[t]$ are the vertices of $T$ that are not in $T_r$. Denote by $\lambda_T(v_1,\ldots, v_k)$ the event that all edges emanating from $v_1, \ldots, v_{t}$ which do not belong to $T(v_1,\ldots,v_k)$ are \emph{not} in the UST $\T_n$. 
We have
$$\Prob (B_{\T_n}(X, r) \cong T) = {1 \over |V(G_n)| |\stab_T|} \hspace{-.3cm} \sum_{\substack{(v_1,\ldots, v_k) \\ T\textrm{-compatible}}} \hspace{-.5cm} \Prob \big ( T(v_1,\ldots,v_k) \subset \T_n \, , \, \lambda_T(v_1,\ldots, v_k) \big ) \, .$$

By Kirchhoff's formula \eqref{eq:kirchhoff} together with spatial Markov's property (\cref{prop:spatialMarkov}) for any $T$-compatible tuple $(v_1,\ldots,v_k)$, we have that 
\be\label{eq:mainthmopenedges} \Prob(T(v_1,\ldots,v_k) \subset \T_n) = \prod_{i=2}^k \Reff(v_i \lr \{v_1,\ldots, v_{i-1}\}) \, .\ee

By the spatial Markov property (\cref{prop:spatialMarkov}), the probability of the event $\lambda_T(v_1,\ldots, v_k)$ conditioned on $T(v_1,\ldots,v_k) \subset \T_n$ is the probability that all edges emanating from $\{v_1,\ldots, v_t\}$ that do not belong to $T(v_1,\ldots,v_k)$ are not in a UST of the graph $G/(v_1,\ldots,v_k)$ in which the edges $T(v_1,\ldots,v_k)$ are contracted and loops erased. The degree in this graph of the the vertex corresponding to $(v_1,\ldots,v_k)$ is at least $kd-k^2$ and at most $kd$. 

Denote by $\{e_1,\ldots, e_L\}$ the set of edges of $G$ that have precisely one endpoint in $\{v_1,\ldots, v_t\}$ and do not belong to $T(v_1,\ldots,v_k)$. Since all degree are $d$ we have that $L\geq td - k^2$. For any $\ell$ satisfying $1 \leq \ell \leq L$ we have 
\begin{align*} \Prob ( e_\ell \not \in \T_n \mid \{e_1,\ldots,e_{\ell-1}\} \cap \T_n = \emptyset \, &, \, T(v_1,\ldots,v_k) \subset \T_n) \\ &\leq 1 - {1 \over d-k+1} - {1 \over kd - k^2 - \ell +1} \, ,\end{align*}
since if $e_\ell=(v_i,u)$ for some $1 \leq i \leq k$ and $u\not \in \{v_1,\ldots,v_k\}$, then the degree of $u$ in $G/(v_1,\ldots,v_k) - \{e_1,\ldots,e_{\ell-1}\}$ is at least $d-k$ (since $G_n$ is simple) and the degree of $v_i$ in the same graph is at least $kd - k^2 - \ell+1$, and the estimate follows by the spatial Markov property and our deterministic lower bound \eqref{eq:reslowerbound} on the effective resistances between two vertices.

We apply this estimate sequentially over $\{e_1,\ldots, e_L\}$, use the fact that $L\geq td - O(1)$ and obtain that
\begin{eqnarray*}\label{eq:closedupperbound} \Prob \big ( \lambda_T(v_1,\ldots, v_k) \mid T(v_1,\ldots,v_k) \subset \T_n \big ) &\leq& \prod_{\ell=1}^{L} \Big ( 1 - {1 \over d-k+1} - {1 \over dk - k^2 - \ell +1}  \Big ) \\ &\leq& \exp \Big ( -\sum_{\ell=1}^{td-k^2} \big ( {1 \over d} + {1 \over dk - \ell} \big ) \Big ) \\ &\leq& (1+O(d^{-1})) {e^{-t} (k - t) \over k} \, , \end{eqnarray*}
where the second inequality is a straightforward computation with harmonic series. The constants in the $O$-notation depend on $k$. Thus,

$$\Prob (B_{\T_n}(X, r) \cong T) \leq {(1+O(d^{-1})) e^{-t} (k - t) \over k |V(G_n)| |\stab_T|} \hspace{-.5cm} \sum_{\substack{(v_1,\ldots, v_k) \\ T\textrm{-compatible}}} \prod_{i=2}^k \Reff(v_i \lr \{v_1,\ldots, v_{i-1}\}) \, .$$

We now show that the probability that $B_{\T_n}(X, r-1)$ contains a vertex that is not good is negligible. Indeed, by \cref{lem:highrescount} the number of such vertices is at most $Cn/\log d$ and so by \cref{lem:stayawayfrombad} the probability that $B_{\T_n}(X, r) \cong T$ and that there exists a vertex of $B_{\T_n}(X, r-1)$ touching such a vertex is bounded by ${C r k^{r-1} \over \log d} = o(1)$ since $d(n) \to \infty$ and $k$ and $r$ are fixed. Hence by \cref{lem:openedges} we get
\be\label{eq:mainineq} \Prob (B_{\T_n}(X, r) \cong T) \leq { e^{-t} (k - t) \over |\stab_T|} + o(1) \, .
\ee


Denote by $T_M$ the set of rooted trees $T$ with $|V(T)|\leq M$, viewed up to root preserving graph isomorphism. Let $\e>0$ be arbitrary, apply \cref{thm:tightness} to obtain a number $M_1=M_1(\e)<\infty$ so that for all $n$ 
$$ \sum _{T \in T_{M_1}}\Prob ( B_{\T_n}(X, r) \cong T ) \in [1- \e,1] \, .$$
By \eqref{eq:poissongoal} there exists $M_2=M_2(\e)<\infty$ such that 
$$ \sum _{T \in T_{M_2}} { e^{-t} (k - t) \over |\stab_T|} \in [1-\e,1] \, .$$
We put $M=\max(M_1, M_2)$.  By \eqref{eq:mainineq} we learn that there exists $N=N(M,\e)$ such that for any $T\in T_M$ and any $n\geq N$ one has
$$ \Prob (B_{\T_n}(X, r) \cong T) \leq { e^{-t} (k - t) \over |\stab_T|} + {\e \over |T_M|} \, .$$
If two positive sequences $\{a_\ell\}_{\ell=1}^N$ and $\{b_\ell\}_{\ell=1}^N$ satisfy $a_\ell \leq b_\ell$ for all $\ell\in[N]$ as well as $\sum_{\ell=1}^N a_\ell \in [1-\e,1]$ and $\sum_{\ell=1}^N b_\ell \in [1-\e,1+\e]$, then $\sum_{\ell=1}^N |a_\ell-b_\ell| \leq 2\e$. Hence 
$$ \sum_{T\in T_M} \Big | \Prob (B_{\T_n}(X, r) \cong T) -   {e^{-t} (k - t) \over |\stab_T|} 
\Big | \leq 2\e \, ,$$
concluding our proof. \qed \\

\section{The UST on high degree almost regular graphs}

In this section we describe the necessary changes to the proof of \cref{thm:mainthm2} in order for it to work when $\{G_n\}$ is a sequence of high degree almost regular graphs with $d(n)\to\infty$ (see \cref{def:almost}), that is, in order to prove \cref{thm:mainthm2Robust}. Furthermore, we prove the quenched version \cref{thm:mainthm2quenched} which relies on \cref{thm:mainthm2Robust} (even in the purely regular setting of \cref{thm:mainthm2}).

We will need to take into account the rate of decay of the $o(1)$ terms in \cref{def:almost}. Therefore we assume for the rest of this section that $\{G_n\}$ is a sequence of high degree almost regular graphs such that
\be \label{eq:ass1} \Big | \big \{ v \in V(G_n): | \deg(v) - d(n) | \leq \delta_n d(n) \big \} \Big | \geq (1-\delta_n)|V(G_n)| \, ,\ee 
and
\be \label{eq:ass2} \Big |\sum_{v\in V(G_n)} \deg(v) - d(n)|V(G_n)| \Big | \leq \delta_n d(n)|V(G_n)| \, ,\ee
where $\delta_n = o(1)$ is some non-negative sequence tending to $0$. Without loss of generality we assume that $d(n)^{-1} = O(\delta_n)$; otherwise we take the sequence $\max(\delta_n, d(n)^{-1})$.

For such a sequence we have that
\be\label{eq:uniformstationaryRobust} \sum_{v \in V(G_n)} \Big | {\deg(v) \over \sum_v \deg(v)} - {1 \over |V(G_n)|} \Big | = O(\delta_n) \, ,\ee
or in other words, we may couple the uniformly drawn vertex with a vertex drawn according to the stationary distribution so that they are equal with probability at least $1-O(\delta_n)$, see \cite[Proposition 4.7]{LPW:Mixing}.

\subsection{Almost regular versions of \cref{sec:Foster,sec:tightness}}
Let $X_0$ be a vertex drawn according to the stationary distribution and let $X_1$ be a random neighbor of it so that $\{X_0,X_1\}$ is a uniformly drawn edge of $G_n$. Then Foster's Theorem gives the analogue of \eqref{eq:fostereff} for high degree almost regular graphs
$$ \Exp \Reff(X_0 \lr X_1) = {2 + O(\delta_n) \over d(n)} \, .$$
\begin{manuallemma}{3.1'} \label{lem:highrescountRobust} {\em Let $\{G_n\}$ be a high degree almost regular graph satisfying \eqref{eq:ass1} and \eqref{eq:ass2}.  For any $\e>2/d$ the number of edges $e=(x,y)\in E(G_n)$ with $\Reff(x \lr y) \geq \e$ is at most ${(1+O(\delta_n))|V(G_n)| \over \e d - 2}$.}
\end{manuallemma}
\begin{proof} Similarly to the proof of \cref{lem:highrescount}, we denote by $N_\e$ the number of such edges. For every edge $(x,y)$ for which the degree of both its vertices is at most $(1+\delta_n)d(n)$ we bound $\Reff(x \lr y) \geq {2 - O(\delta_n) \over d}$ using \eqref{eq:reslowerbound}. By \eqref{eq:ass1} and \eqref{eq:ass2} the number of such edges is at least $(1-O(\delta_n))|V(G_n)|d/2$, so by Foster's Theorem \eqref{thm:foster}
$$ |V(G_n)|-1 \geq N_\e \e + \Big ( {(1-O(\delta_n))|V(G_n)|d \over 2} - N_\e \Big ) { 2-O(\delta_n) \over d} \, ,$$
giving the required upper bound on $N_\e$. 
\end{proof}

We use \cref{lem:highrescountRobust} to prove tightness in the almost regular setting. 

\begin{manualtheorem}{4.2'}[Tightness] \label{thm:tightnessRobust} {\em Let $\{G_n\}$ be a high degree almost regular graph sequence. Let $\T_n$ be a uniformly drawn spanning tree of $G_n$ and let $X$ be a random vertex chosen according to the stationary distribution of $G_n$. Then for any integer $r\geq 0$ we have
$$ \lim _{M \to \infty} \sup_n \Prob \Big ( |B_{\T_n}(X,r)| \geq M \Big ) = 0 \, .$$}
\end{manualtheorem}
\begin{proof} The proof of \cref{thm:tightness} applies verbatim with \cref{lem:highrescountRobust} replacing the use of \cref{lem:highrescount}.
\end{proof}

Next we alter the statement of \cref{lem:fosterRW} to the following.

\begin{manuallemma}{3.2'}\label{lem:fosterRWRobust} {\em Let $\{G_n\}$ be a high degree almost regular graph satisfying \eqref{eq:ass1} and \eqref{eq:ass2} and let $(X_0,\ldots,X_k)$ be a $k$-step random walk on $G_n$, with $k\geq 1$, starting from a random vertex $X_0$ drawn according to the stationary distribution. Then
$$ \Exp \Reff(X_0 \lr X_k) \leq {2 + O(\delta_n) \over d}\, .$$}
where the implicit constant may depend on $k$.
\end{manuallemma}
\begin{proof}
The alterations required in the proof of \cref{lem:fosterRW} are changing the stationary mass of $x_0$ and replacing powers of $d$ into the corresponding products of degrees. So that \eqref{eq:x1return} now becomes
$$ \sum_{\substack{(x_1,\ldots,x_k) \\ x_i \neq x_0 \forall i=1,\ldots,k}} {1 \over \prod_{i=0}^{k-1} \deg(x_i) } \big [ \E_{x_k} \tau_{x_0} + k \big ] \leq {1 \over \pi(x_0)} \, ,$$
where $\pi(x_0) = \deg(x_0) / \sum_{v\in V(G_n)} \deg(v)$ is the stationary mass of $x_0$. The analogue of \eqref{eq:xkhitx1} is now
$$ \E \E _{X_k} \tau_{x_0} \leq {1 \over \pi(x_0)} + \sum_{i=1}^{k-1} \sum_{\substack{(x_1,\ldots,x_{i}) \\ x_i=x_0}} {1 \over \prod_{j=0}^{i-1} \deg(x_j) } \hspace{-.3cm} \sum_{\substack{(x_{i+1}, \ldots, x_k) \\ x_j \neq x_0 \forall j=i+1,\ldots,k}} {1 \over \prod_{j=i}^{k-1} \deg(x_j) } \E_{x_k} \tau_{x_0} \, .$$
Hence plugging in the previous estimate yields
$$ \E \E _{X_k} \tau_{x_0} \leq {1 \over \pi(x_0)} + {1 \over \pi(x_0)} \sum_{i=1}^{k-1} \Prob(X_i=x_0 \mid X_0= x_0) \, .$$
We average this inequality over $x_0$ according to the stationary measure of $G_n$ and get that
$$ \E \E _{X_k} \tau_{X_0} \leq |V(G_n)|+ \sum_{i=1}^{k-1} \sum_{x_0}  \Prob(X_i=x_0 \mid X_0= x_0) \, .$$
For each $i\in [k-1]$, the corresponding term in the sum on the right hand side, divided by $|V(G)|$, is just the probability that $\Prob(X_i=X_0)$ when $X_0$ is a uniformly chosen vertex. It is easier to bound the return probability when $X_0$ is a stationary vertex rather than a uniform vertex. By \eqref{eq:uniformstationaryRobust} the difference between the two probabilities is $O(\delta_n)$. If $X_0$ is stationary, then $X_{i-1}$ is also a stationary vertex, hence the probability that it is a vertex with degree smaller than $d/2$ is $O(\delta_n)$ by \eqref{eq:ass1}, and if that does not occur the probability of moving to $X_0$ in the next step is at most $2/d=O(\delta_n)$. We deduce that 
$$ \E \E _{X_k} \tau_{X_0} \leq (1+O(\delta_n))|V(G_n)| \, ,$$
and the proof continues precisely as in \cref{lem:fosterRW} to give the desired result. 
\end{proof}


\begin{manualcor}{3.3'} \label{cor:highrescountRWRobust} {\em Let $\{G_n\}$ be a high degree almost regular graph satisfying \eqref{eq:ass1} and \eqref{eq:ass2} and let $(X_0,\ldots,X_k)$ be a $k$-step random walk on $G_n$, with $k\geq 1$, starting from a random vertex $X_0$ drawn according to the stationary distribution. Then for any $\e > 2/d$ we have
$$\Prob \big ( \Reff(X_0 \lr X_k) \geq \e \big ) \leq {O(\delta_n) \over \e d - 2 } \, .$$}
\end{manualcor}
\begin{proof}
Denote the probability that $\Reff(X_0 \lr X_k) \geq \e$ by $p$. Since $X_0$ is stationary so is $X_k$ hence by \eqref{eq:ass1} the probability that they both are vertices of degree $(1\pm O(\delta_n))d(n)$ is at least $1-O(\delta_n)$. In this case we bound the resistance between them from below by ${2-O(\delta_n) \over d(n)}$ using \eqref{eq:reslowerbound}. Thus, by \cref{lem:fosterRWRobust} we obtain 
$$ {2 + O(\delta_n) \over d} \geq \Exp \Reff(X_0 \lr X_k) \geq \e p + (1-p-O(\delta_n)){2 - O(\delta_n) \over d} \, ,$$
and rearranging gives the result.
\end{proof}

We are now ready to state the analogue of \cref{thm:GeneralFoster}. Given a fixed finite rooted tree $T$ with $k\geq 3$ vertices, the random $k$-tuple $(X_1,\ldots,X_k)$ of vertices is drawn exactly as described above \cref{thm:GeneralFoster} with the only exception that $X_0$ is now a stationary vertex. 

\begin{manualtheorem}{3.4'} \label{thm:GeneralFosterRobust} {\em Let $\{G_n\}$ be a high degree almost regular graph satisfying \eqref{eq:ass1} and \eqref{eq:ass2}. Let $\{X_1,\ldots,X_k\}$ be a random $k$-tuple drawn as described above. Then there exists $C=C(k)>0$ such that with probability at least $1 - O(\delta_n^{1/2})$ we have that $(X_1,\ldots,X_k)$ are $T$-compatible and 
$$ \Big | \Reff \big (X_k \lr \{X_1, \ldots, X_{k-1}\} \big ) - {k \over (k-1) d} \Big | \leq {C\delta_n^{1/2} \over d} \, .$$}
\end{manualtheorem}
\begin{proof} We closely follow the the proof of \cref{thm:GeneralFoster}. For $i\neq j$ as in the beginning of that proof, we have by \cref{cor:highrescountRWRobust} that

$$ \Prob \Big ( \Reff(X_i \lr X_j) \geq {2 + \sqrt{\delta_n} \over d} \Big ) = O(\delta_n^{1/2}) \, .$$
By the union bound
$$\Prob \Big ( \exists i\neq j \in [k] \quad \Reff(X_i \lr X_j) \geq {2 + \sqrt{\delta_n} \over d} \Big ) = O(\delta_n^{1/2}) \, .$$
Furthermore, since $X_i$ are stationary, by \eqref{eq:ass1} and \eqref{eq:uniformstationaryRobust}, with probability at least $1-O(\delta_n)$ the degrees of $X_1,\ldots,X_k$ are all $(1+O(\delta_n))d(n)$, hence the resistance between any distinct pair is at least ${2-O(\delta_n) \over d}$. Furthermore, as before, the probability that $X_1,\ldots,X_k$ are all distinct is $1-O(d^{-1})$ which is $1-O(\delta_n)$. We conclude that with probability at least $1-O(\delta_n^{1/2})$ for all $i,j\in[k]$ with $i\neq j$ we have
$$ \Big | \Reff(X_i \lr X_j) - {2 \over d} \Big | \leq {O(\delta_n^{1/2}) \over d} \, ,$$
from which \cref{thm:reff1tomany} implies that desired result.
\end{proof}

\subsection{Proof of main theorem, extended version}
In the proofs of the rest of this section we will need to discard undesirable $T$-compatible $k$-tuples such as those that have atypical vertex degrees. To that aim we introduce the following definitions. Recall that our convention is that the vertices $[t]$ where $t<k$ are the vertices of $T$ at graph distance at most $r-1$ from the root, while $\{t+1,\ldots, k\}$ are the vertices at graph distance precisely $r$ from the root.

\begin{defin} \label{def:robustdefs}  Let $\{G_n\}$ be a high degree almost regular graph sequence satisfying \eqref{eq:ass1} and \eqref{eq:ass2}. 
\begin{enumerate}
\item We say that a $T$-compatible $k$-tuple $(v_1,\ldots,v_k)$ has {\bf typical degrees} if 
$$ (1-\delta_n) d(n) \leq \deg(v_i)\leq (1+\delta_n) d(n) \qquad \forall i \in [t] \, .$$
\item We say that a $T$-compatible $k$-tuple $(v_1,\ldots,v_k)$ which has typical degrees has {\bf typical neighbor degrees} if for each $i \in [t]$ the number of neighbors of $v_i$ which have degree at least $(1-\delta_n) d(n)$ is at least $(1-\sqrt{\delta_n}) d(n)$. 

\item We say that a $T$-compatible $k$-tuple $(v_1,\ldots,v_k)$ is ${\bf good}$ if none of the vertices $v_1,\ldots,v_t$ are incident to edges with resistance at least ${\delta_n^{-{1 \over 4(k-1)}} \over d}$ across them. 
\end{enumerate}
\end{defin}

In the rest of this section all implicit constants in the $O$-notation depend on $k$.

\begin{claim} \label{claim:whpgood} Let $T$ be a fixed rooted tree with $k$ vertices, and $(X_1,\ldots,X_k)$ be a $T$-compatible $k$-tuple drawn as described above \cref{thm:GeneralFosterRobust}. Then with probability at least $1-O(\sqrt{\delta_n})$ the tuple $(X_1,\ldots,X_k)$ has typical degrees and typical neighbor degrees. Also, with probability $1-O(\delta_n^{1/4(k-1)})$ it is good.
\end{claim}
\begin{proof} Since $X_i$ is distributed according to the stationary distribution for $i\in[t]$ it suffices to prove that $X_1$ satisfies the requirement of \cref{def:robustdefs} and use the union bound to obtain the desired result. 

Indeed, firstly, by \eqref{eq:ass1} and \eqref{eq:uniformstationaryRobust} we learn that with probability $1-O(\delta_n)$ the vertex $X_1$ has degree within $(1\pm\delta_n)d(n)$. Secondly, by \cref{lem:highrescountRobust} the probability that $X_1$ is not good is $O(\delta_n^{1/4(k-1)})$. Thirdly, denote by $V_1 \subset V(G_n)$ the set of vertices with degree within $(1\pm \delta_n)d(n)$ and by $V_2 \subset V_1$ the set of vertices in $V_1$ such that at least $(1-\sqrt{\delta_n})d(n)$ of their neighbors are in $V_1$. As before we have that $\Prob(X_2 \in V_1) = 1-O(\delta_n)$. On the other hand, 
$$ \Prob(X_2 \in V_1 \mid X_1 \in V_1 \setminus V_2) \leq 1 - O(\sqrt{\delta_n}) \, .$$
Thus
$$ 1-O(\delta_n) \leq \Prob(X_2 \in V_1) \leq (1-O(\sqrt{\delta_n})) \Prob(X_1 \in V_1 \setminus V_2) + 1 - \Prob(X_1 \in V_1 \setminus V_2) \, ,$$
and we deduce that $\Prob(X_1 \in V_1 \setminus V_2) = O(\sqrt{\delta_n})$ concluding the proof.
\end{proof}

The corresponding analogue of \cref{cor:generalhighrescount} requires that we consider $k$-tuples having typical degrees. 

\begin{manualcor}{3.6'} \label{cor:generalhighrescountRobust} {\em Let $\{G_n\}$ be a high degree almost regular graph satisfying \eqref{eq:ass1} and \eqref{eq:ass2}. Denote by $N$ the number of $k$-tuples $(v_1,\ldots,v_k)$ that are $T$-compatible, have typical degrees, and such that 
$$ \Big | \Reff ( v_k \lr \{v_1,\ldots,v_{k-1}\}) - {k \over (k-1) d} \Big | \geq {C \delta_n^{1/2} \over d} \, ,$$
where $C$ is the constant from \cref{thm:GeneralFosterRobust}. Then
$$ N = O(\delta_n^{1/2} |V(G_n)| d^{k-1}) \, .$$ }
\end{manualcor}
\begin{proof} For any $T$-compatible $k$-tuple $(v_1,\ldots,v_k)$ that has typical degrees the probability that $(X_1,\ldots,X_k)$ equals $(v_1,\ldots,v_k)$ is ${1+O(\delta_n) \over |V(G_n)| d(n)^{k-1}}$. Furthermore, by \cref{claim:whpgood}, the probability that $(X_1,\ldots,X_k)$ have typical degrees is $1-O(\sqrt{\delta_n})$, hence the total number of $T$-compatible $k$-tuples that have typical degrees is $(1+O(\sqrt{\delta_n}))|V(G_n)| d^{k-1}$.

Lastly, by \cref{claim:whpgood} we learn that the assertion of \cref{thm:GeneralFosterRobust} continues to hold when we condition on $(X_1,\ldots,X_k)$ to have typical degrees and the desired assertion follows.
\end{proof}

We may now proceed to the proof of the main theorem in the almost regular setting. We first state the analogue of \cref{lem:openedges}.

\begin{manuallemma}{5.1'}\label{lem:openedgesRobust} {\em Let $\{G_n\}$ be a high degree almost regular graph satisfying \eqref{eq:ass1} and \eqref{eq:ass2}. Then
\be\label{eq:openedgesRobust} {1 \over |V(G_n)|} \sum_{\substack{(v_1,\ldots, v_k) \\ T\mathrm{-compatible} \\ \mathrm{good,\  typical\ degrees}}} \prod_{i=2}^k \Reff(v_i \lr \{v_1,\ldots, v_{i-1}\}) \leq k + O(\delta_n^{1/4}) \, .\ee }
\end{manuallemma}
\begin{proof} We closely follow the proof of \cref{lem:openedges} and prove by induction on $k$. By \cref{cor:generalhighrescountRobust} the number of $T$-compatible $k$-tuples $(v_1,\ldots,v_k)$ for which 
\be\label{eq:resineqR} \Reff(v_k \lr \{v_1,\ldots,v_{k-1}\}) \geq {k/(k-1) + C\sqrt{\delta_n} \over d} \ee
is $O(\delta_n^{1/2} |V(G_n)| d^{k-1})$. For such tuples, we bound each term in the product by ${\delta_n^{-1/4(k-1)} \over d}$. Thus we may bound the sum over such tuples (with the $|V(G_n)|^{-1}$ factor) from above by
$$ O(\delta_n^{1/2} d^{k-1}) \times \Big ( {\delta_n^{-1/4(k-1)} \over d} \Big )^{k-1} = O(\delta_n^{1/4}) \, .$$

For all other tuples we have the opposite inequality at \eqref{eq:resineqR}. Thus we may bound the last term in the product by this, sum it over the at most $(1+\delta_n)d$ possible choices of $v_k$ and obtain that the left hand side of \eqref{eq:openedgesRobust} is bounded above by
\begin{eqnarray*} O(\delta_n^{1/4}) + {k+O(\sqrt{\delta_n}) \over k-1} {1 \over |V(G_n)|} \sum_{\substack{(v_1,\ldots, v_{k-1}) \\ T\setminus\{v_k\}\mathrm{-compatible} \\ \mathrm{good,\  almost\ regular}}} \prod_{i=2}^{k-1} \Reff(v_i \lr \{v_1,\ldots, v_{i-1}\}) \, .\end{eqnarray*}
We apply our induction hypothesis to the sum on the right hand side, collect the telescoping terms and get the desired result. 
\end{proof}

\noindent {\bf Proof of \cref{thm:mainthm2Robust}.} Recall the definitions of the events $T(v_1,\ldots,v_k)$ and $\lambda_T(v_1,\ldots,v_k)$ from the proof of \cref{thm:mainthm2}. We still have that
$$ \Prob (B_{\T_n}(X, r) \cong T) = {1 \over |V(G_n)| |\stab_T|} \sum_{\substack{(v_1,\ldots, v_k) \\ T\textrm{-compatible}}} \Prob ( B_{\T_n}(v_1, r) = T(v_1,\ldots,v_k) ) \, .$$
To proceed with the proof we need to discard a larger set of $k$-tuples than we did in the proof of \cref{thm:mainthm2} and to do so earlier. By \cref{lem:stayawayfrombad} and \cref{claim:whpgood} the probability that $B(X,r)\cong T$ and $B(X,r-1)$ contains a vertex that is not of typical degree or not good or does not have typical neighbor degree is $o(1)$. Thus,

\begin{align*} \Prob (B_{\T_n}(X, r) \cong T) &= {1 \over |V(G_n)| |\stab_T|} \hspace{-.9cm} \sum_{\substack{(v_1,\ldots, v_k) \\ T\textrm{-compatible} \\ \textrm{good, typical degrees,} \\\textrm{typical neighbor degrees}}}\hspace{-1cm} \Prob ( T(v_1,\ldots,v_k) \subset \T_n \, , \, \lambda_T(v_1,\ldots,v_k) ) \\ &+ o(1) \, .\end{align*}
The analysis performed in the proof of \cref{thm:mainthm2} shows that for any $T$-compatible $k$-tuple $(v_1,\ldots,v_k)$ that has typical degrees and typical neighbor degrees we have
$$ \Prob ( \lambda_T(v_1,\ldots,v_k) \mid T(v_1,\ldots,v_k) \subset \T_n) \leq (1+o(1)) {e^{-t} (k-t) \over k} \, .$$
Since \eqref{eq:mainthmopenedges} holds, by \cref{lem:openedgesRobust} and the above we get that
$$ \Prob (B_{\T_n}(X, r) \cong T) \leq {e^{-t}(k-t) \over |\stab_T|} + o(1) \, .$$
Now the same proof as in \cref{thm:mainthm2}, below \eqref{eq:mainineq}, can now be used verbatim, with the exception that \cref{thm:tightnessRobust} takes the role of \cref{thm:tightness}. \qed 

\subsection{Proof of \cref{thm:mainthm2quenched}.}


The proof by a second moment argument. By \cref{thm:mainthm2Robust} we have that 
$$ \Exp Y_n(T) = (1+o(1))\frac{|V(G_n) |T_r| e^{-|V(T)|+|T_r|}}{|\stab_T|} \, .$$ 
The second moment can be expressed as
$$ \Exp Y_n(T)^2 = \sum_{u,v\in V(G_n)} \Prob(B_{\T_n}(v,r)\cong T \and B_{\T_n}(u,r)\cong T) \, .$$
We split the last sum to two according to whether the intersection $B_{\T_n}(v,r) \cap B_{\T_n}(u,r)$, viewed as a vertex subset, is empty or not. If the intersection is non-empty, then $u \in B_{\T_n}(v,2r)$. Hence we bound
\begin{eqnarray}\label{eq:splitsum} \Exp Y_n(T)^2 &\leq& \sum_{u,v} \Prob(B_{\T_n}(v,r)\cong T , B_{\T_n}(u,r)\cong T , B_{\T_n}(v,r) \cap B_{\T_n}(u,r) = \emptyset) \nonumber \\ &+& \sum_{v} \Exp\big [ |B_{\T_n}(v,2r)| \big ] \, .
\end{eqnarray}
\cref{thm:tightnessRobust} immediately yields that the second term in \eqref{eq:splitsum} is $o(|V(G_n)|^2)$ so we focus on estimating the first sum of the above inequality. To that aim, we condition on $B_{\T_n}(v,r)\cong T$ and on $B_{\T_n}(v,r)$ itself. That is, in the terminology of the proof of \cref{thm:mainthm2}, we condition the $T$-compatible $k$-tuple $(v_1,\ldots,v_k)$ and on the edges $T(v_1,\ldots,v_k)$ being in the UST $\T_n$ and on the event $\lambda_T(v_1,\ldots,v_k)$, i.e., that all edges touching $\{v_1,\ldots,v_t\}$, except for those in $T(v_1,\ldots,v_k)$, are not in $\T_n$. Conditioned on this information, the UST $\T_n$ restricted to the unconditioned edges is distributed as a UST on the graph obtained by $G_n$ by contracting the vertices $\{v_1,\ldots,v_k\}$ to a single vertex and erasing edges touching $\{v_1,\ldots,v_t\}$. Denote the resulting graph by $G_n'$.

It is not hard to verify that $G_n'$ is high degree almost regular. Indeed, since $G_n$ is simple, the degrees of all vertices of $G_n$ except $\{v_1,\ldots,v_k\}$ have dropped by at most $k$. The degree of the conjoined vertex $\{v_1,\ldots,v_k\}$ is at most $kn$ and at least $1$ and the total number of edges that were erased from $G_n$ is at most $t n = o(dn)$. Thus $G_n'$ satisfied \cref{def:robustdefs} perhaps with a slightly larger $\delta_n$ than the one of $G_n$, but still $o(1)$. Denote by $\T_n'$ the UST on $G_n'$. Then \cref{thm:mainthm2Robust} gives
$$ \sum_{u\in V(G_n')}  \Prob(B_{\T_n'}(u,r)\cong T) = (1+o(1))\frac{|V(G_n)||T_r| e^{-|V(T)|+|T_r|}}{|\stab_T|} \, . $$
Therefore, the first sum in \eqref{eq:splitsum} is just $(1+o(1)) [ \Exp Y_n(T) ] ^2$ while the second sum is $o(\Exp Y_n(T))$. We deduce that $\Exp Y_n(T)^2 = (1+o(1)) [\Exp Y_n(T)^2]$, or in other words, the variance of $Y_n(T)$ is $o([\Exp Y_n(T)]^2)$ and the assertion of the theorem follows by Chebychev's inequality. \qed

\section{Concluding remarks and open problems}
\subsection{Maximal diameter of the UST on regular graphs}\label{sec:maxdiamquestion} The diameter of the UST, i.e., the maximal graph distance between two vertices, is the most natural ``global'' property of the UST. We cannot hope this quantity has a universal behavior only under the assumption of regularity. Indeed, the complete graph on $d+1$ vertices is a $d$-regular graph in which the diameter of the UST is of order $\sqrt{d}$. In fact, in \cite{MNS19} it shown that that the diameter of the UST on various ``high-dimensional'' (such as regular expanders, the hypercube and $d$-dimensional tori for $d>4$) is of order $\sqrt{|V(G)|}$.

 On the other hand, take $m$ disjoint copies $K_1,\ldots, K_m$ of complete graphs on $d+1$ vertices and for each $i\in[m]$ let $(x_i,y_i)$ be an edge of $K_i$. Add the edges $(y_i,x_{i+1})$ for each $i\in[m-1]$ to form the connected graph $G$ --- it is easily seen that the diameter of the UST on $G$ is of order $n/\sqrt{d}$. (We remark that $G$ is not regular, however, it is not difficult to make small changes to this construction to make it regular, we omit the details.) Our first question is whether this upper bound is best possible. \\

\noindent{\bf Question.} Let $\{G_n\}$ be a sequence of simple, finite, connected $d(n)$-regular graphs with $d(n)\to \infty$. Let $D_n$ be the diameter of a UST of $G_n$. Does it hold that 
$$ \Exp D_n = O\Big ({ |V(G_n)| \over \sqrt{d(n)}} \Big ) \qquad ?$$

\noindent {\bf Remark 1}. In \cite{ANS20} it is shown that if $G$ is a simple, finite, connected graph with minimal degree $\Omega(n)$, then its diameter is $\Theta(\sqrt{n})$. This confirms the question above when $d(n)$ is linear in $|V(G_n)|$. \\

\noindent {\bf Remark 2}. \cref{thm:mainthm} implies that with high probability that diameter of the UST in the setting of the question above is $o(|V(G_n)|)$. In fact, the following much more general statement can be made. We thank Jan Hladk\'y for showing this to us.

\begin{prop}\label{prop:lineardiam} Let $\{G_n\}$ be a sequence of graphs with $|V(G_n)| \to \infty$ such that the uniform spanning tree $\T_n$ of $G_n$ has a local limit $(\T,o)$ that almost surely has no bi-infinite paths. Then, for any fixed $\e>0$ one has 
$$ \lim_{n \to \infty} \,\,\Prob \big ( D_n \geq \e V(G_n) \big ) = 0 \, ,$$
where $D_n$ is the diameter of $\T_n$.
\end{prop}
\begin{proof} Assume by contradiction that there exists some fixed positive number $\e_1, \e_2$ such that for infinitely many $n's$ we have
$$ \Prob \big ( D_n \geq \e_1 V(G_n) \big ) \geq \e_2 \, .$$
By passing to a subsequence we may assume without loss of generality that this happens for every $n$. Let $X$ be a uniform vertex of $G_n$. Then for any  
positive integer $r$, the probability that $B_{\T_n}(X,r)$ contains two disjoint paths of length $r$ emanating from $X$ is at least $\e_2 (\e_1 - 2r /|V(G_n)|)$ which is at least $\e_1 \e_2/2>0$ when $n$ is large enough. This contradicts the fact that $(\T,o)$ has no bi-infinite paths almost surely.
\end{proof}

\begin{cor} \label{cor:lindiam} The conclusion of \cref{prop:lineardiam} hold for any sequence $\{G_n\}$ of simple, finite, connected $d(n)$-regular graphs with $d(n)\to \infty$. 
\end{cor}
\begin{proof} It is well-known that the Poisson$(1)$ Galton-Watson tree conditioned to survive forever has no bi-infinite paths almost surely, see \cite{conceptual95}.
\end{proof}

\subsection{Moments of degrees and graph distance balls}\label{subsec:questionsballs} The average degree of a finite tree is at most $2$ so $\E \deg_{\T_n}(X) \leq 2$ where $\T_n$ is a UST of some finite graph and $X$ is a uniformly drawn vertex. What can be said about higher moments? 

Without the assumption of regularity no higher moments are necessarily bounded as can be seen by taking a star on $n$ vertices. In the class of regular graphs we have the following construction. Assume that $d$ is a large even integer and consider $d/2$ disjoint copies of complete graphs on $d$ vertices. From each complete graph remove an edge, add a new vertex to the graph and connect it to each complete graph with two edges to the two endpoints of the removed edges. In any spanning tree of this graph the degree of the special vertex must be at least $d/2$ and the probability that a uniform vertex is the special one is of order $d^{-2}$. Thus, the $p$-th moment of the degree for any $p>2$ need not be bounded. However, we can use the techniques of this paper to show that the $p$-th moment exists for any $p\in(1,2)$.

\begin{lemma} There exists a constant $C<\infty$ such that for any finite connected regular graph $G$
$$ \Prob( \deg_\T(X) \geq k ) \leq {C \over k^2} \, ,$$
where $\T$ is a UST of $G$ and $X$ is a uniformly drawn vertex.
\end{lemma}
\begin{proof}
Assume that $k\geq 16e$. By \cref{lem:highrescount} there are no more than ${8e|V(G)| \over k}$ edges with effective resistance at least ${k \over 4ed}$ between their endpoints. So the number of vertices which touch at least $k/2$ such edges is at most ${32 e |V(G)| \over k^2}$. Hence the probability that $X$ is such a vertex is at most ${32 e \over k^2}$. If $X$ is not such a vertex and its degree is at least k, then there are at least $k/2$ edges that touch $X$ which have effective resistance at most ${k \over 4ed}$ that are in the UST. The probability of this, by \eqref{eq:negativecorr}, is at most 
$$ {d \choose k/2} (k/4ed)^{k/2} \leq (2ed/k)^{k/2} (k/4ed)^{k/2} \leq 2^{-k} \, ,$$
concluding the proof. \end{proof}

\noindent {\bf Question.} Does there exist a constant $C$ such that for any finite connected regular graph $G$ one has $\E \deg^2_{\T}(X) \leq C$? \\



A very much related quantity is the size of the graph distance ball of radius $r$. In \cref{thm:tightness} we have proved that its size is tight. The case $r=1$ of this statement is trivial since  $|B_\T(X,1)| = \deg_\T(X)$ and so their first moment is at most $2$. When $r=2$ the question above about $\E \deg^2_\T(X)$ is equivalent to asking whether $\E |B_\T(X,2)|$ is bounded. Indeed, it is easy to see by the mass transport principle \cite[Chapter 8]{LyPe:ProbabilityTrees} that $\E \deg^2_\T(X) = \E |B_\T(X,2)|$ -  (each vertex transports its degree in $\T$ to all of its neighbors). The mean of $|B_\T(X,3)|$ can already be unbounded --- we leave this as an exercise to the reader. 


\section*{Acknowledgements} This research is supported by ERC starting grant 676970 RANDGEOM and by ISF grant 1207/15 and 1294/19. We wish to thank Jan Hladk\'y for many useful conversations and for his permission to include his proof of \cref{prop:lineardiam}. We also thank Matan Shalev for finding several errors in a previous version of this manuscript.

\footnotesize{
\bibliographystyle{abbrv}
\bibliography{bibl}
 }

\end{document}